\documentclass[11pt]{article}
\usepackage{amsfonts,amsmath,amssymb,amsthm, amscd, subcaption}
\usepackage{graphicx}
\numberwithin{equation}{section}

\newtheorem{theorem}{Theorem}[section]
\newtheorem{prop}[theorem]{Proposition}
\newtheorem{lemma}[theorem]{Lemma}
\newtheorem{cor}[theorem]{Corollary}

\theoremstyle{definition}
\newtheorem{definition}[theorem]{Definition}
\newtheorem{example}[theorem]{Example}
\newtheorem{remark}[theorem]{Remark}

\newcommand{\D}{\Delta}

\def\<{{\langle}}
\def\>{{\rangle}}
\def\G{{\Gamma}}

\def\d{{\delta}}

\def\d{{\delta}}
\def\Z{\mathbb Z}
\def\Q{\mathbb Q}
\def\R{\mathbb R}

\def\S{{\mathbb S}}

\def\Co{\mathbb C}
\def\B{{\cal B}}
\def\D{\mathbb D}

\def\s{\sigma}
\def\t{\tau}

\def\w{\omega}
\def\Gr{\Bbb G}
\def\L{{\cal L}}
\def\La{{\Lambda}}

\def\Rd{{\cal R}_d}

\def\De{\Delta}
\def\o{\overline}

\def\ni{\noindent} 

\begin{document}

\title{Links in Surfaces and Laplacian Modules }

\author{Daniel S. Silver 
\and Susan G. Williams}

\maketitle 


\begin{abstract}  Laplacian matrices of weighted graphs in surfaces $S$ are used to define module and polynomial invariants of  $\Z/2$-homologically trivial links  in $S \times [0,1]$. Information about virtual genus is obtained.  \bigskip

\end{abstract}

\section{Introduction} \label{Intro} 

The Laplacian matrix of a signed graph is a discrete version of the well-studied Laplacian operator of physics. Its spectrum yields insights about the structure of the graph. (For a survey see \cite{Mo12}). For signed graphs embedded in a closed oriented surface $S$,  information about how the graph winds about $S$ can be used to define a Laplacian matrix $L_G$ with homology coefficients (see, for example, \cite{Ke11}). We introduce the module $\L_G$ presented by $L_G$. Its module order is a polynomial $\De_G$ over the Laurent polynomial ring $\La = \Z[x_1^{\pm1}, y_1^{\pm1}, \ldots, x_g^{\pm1}, y_g^{\pm1}]$, where $g$ is the genus of $S$. The polynomial is the determinant of $L_G$, but we will see that it can also be computed using a simple skein relation.

When the graph $G$ is embedded in $S$ the well-known medial construction associates a checkerboard colored diagram $D \subset S$ for a link $\ell$ in the thickened surface $S \times [0,1]$. 
 In \cite{ILL10} the authors showed that signature, nullity and determinant, classical link invariants, can be defined for $\ell$ using the usual Laplacian matrix, which can be viewed as a generalization of the Goeritz matrix of a link. Here we consider the Laplacian matrix with homological coefficients. We show that the module $\L_G$ it presents is  unaffected by Reidemeister moves. Invariants of links in the thickened surface $S \times [0,1]$ are described. 

Any link $\ell$ in $S \times [0,1]$ can be described by a diagram in $S$. Adding or deleting ``hollow handles" to $S$, avoiding $D$, might produce a surface of smaller genus. The minimum possible genus is the \textit{virtual genus} of $\ell$. In many cases the polynomial $\De_G$ can be used to establish that a surface achieves the minimum possible genus.

\section*{Acknowledgements} The authors are grateful to Seiichi Kamada, Louis H. Kauffman and Christine Ruey Shan Lee for helpful comments.

\section{Laplacian modules and polynomials}

A graph $G$ in a closed, connected oriented surface $S$ with edge and vertex sets $V_G, E_G$, respectively, is \textit{signed} if every edge $e \in E_G$ is labeled with $\s_e = +1$ or $-1$. In order to avoid visual clutter, unlabeled edges will be assumed to have sign $+1$. 

Let $\tilde S$ denote the universal abelian covering space of $S$. We regard the deck transformation group $A(\tilde S)$ multiplicatively. It is isomorphic to $H_1(\tilde S; \Z)$. We fix a symplectic homology basis $\B =
\{ x_1, y_1, \ldots, x_g, y_g \}$, and represent its members by oriented closed curves. Each closed curve in $S$ is labeled by the element of $A(\tilde S)$ that it determines, a monomial in the variables $x_i, y_i$.

Although we work with undirected graphs, it is convenient for the following definition to regard $G$ as directed. There is a standard way to do this: replace each edge $e \in E_G$ (including loops) with a pair of edges joining the same endpoints, each with the same weight as $e$ but with opposite directions.  Any directed path in $S$  determines an element $\phi_P \in A(\tilde S)$, read by traveling along the path in the preferred direction, recording $y_k^{\pm 1}$ (resp. $x_k^{\pm 1}$) as we cross the homology basis  curve that is labeled $x_k$ (resp. $y_k$), the exponent determined as in Figure \ref{path}. In particular, each directed edge $e \in E_G$ determines a monomial $\phi_e$.

\begin{figure}
\begin{center}
\includegraphics[height=1.5 in]{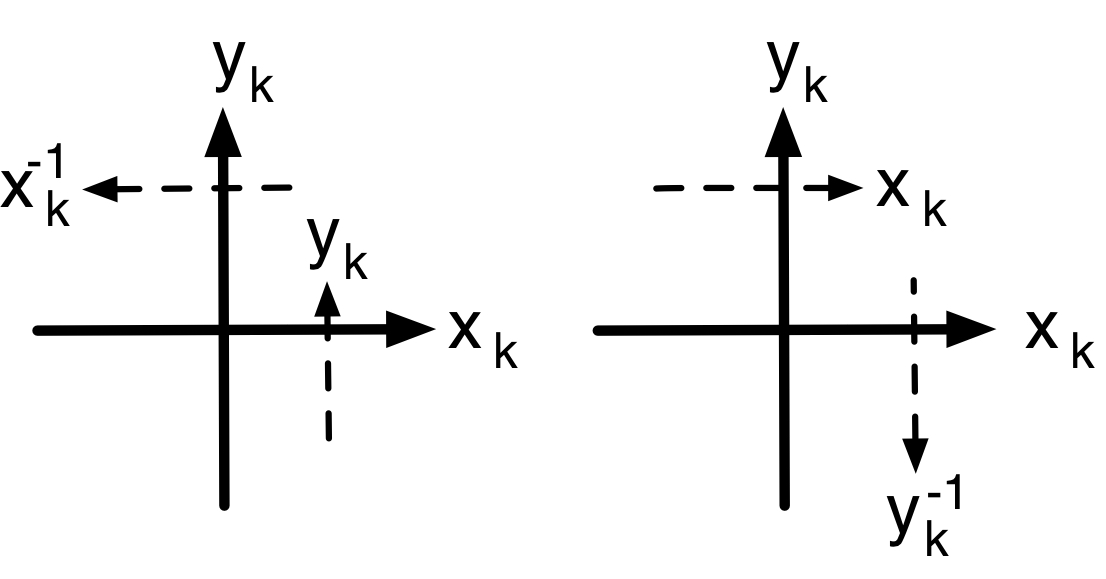}
\caption{Finding the monomial of a path transverse to $G$}
\label{path}
\end{center}
\end{figure}

Assume that $G$ is a signed directed graph with vertex set $V_G= \{v_1, \ldots, v_n\}$. The \textit{adjacency matrix $A_G= (a_{i,j})$ of $G$ (with respect to $\B$)} is the $n \times n$ matrix with non-diagonal entries
$a_{i,j}$ equal to $\sum \s_e \phi_e$, where the summation ranges over all edges from $v_i$ to $v_j$. Diagonal entries of $A_G$ are zero. Denote by $\d_g = (\d_{i,j})$ the $n \times n$ diagonal matrix with $\d_{i, i} = \sum_j a_{i,j}$. 

\begin{definition} The \textit{Laplacian matrix $L_G$ of $G$ (with respect to $\B$)} is $A_G - \d_G$. 
The Laplacian matrix of a signed undirected graph $G$ is the Laplacian matrix of the associated directed graph. \end{definition}

Henceforth the graphs that we consider will be weighted but undirected.

The entries of $L_G$ are elements of the Laurent polynomial ring $\La = \Z[x_1^{\pm1}, y_1^{\pm1}, \ldots, x_g^{\pm1}, y_g^{\pm1}]$. The graph $G$ determines a line bundle with a copy of $\La$ at each vertex. Adopting the terminology of \cite{Ke11}, we call $\phi_e$ the \textit{connection} of the edge $e$.

\begin{definition}\label{lapdef} Let $G$ be a graph in $S$. Its \textit{Laplacian polynomial $\De_G$ (with respect to $\B$)} is the determinant of $L_G$.  \end{definition}

\begin{remark} When $G$ is directed and connections are ignored, the Laplacian matrix $L_G$ is expressible as $F F^t$, where $F$ is the incidence matrix of $G$ and $F^t$ denotes its transpose. This form has had many applications. For example it commonly used in the proof of the Matrix Tree Theorem (see, for example, \cite{Mo12}). In \cite{DKM13} the authors use it to extend the theory of Laplacian matrices and critical groups from graphs to simplicial complexes. There $F$ functions as a boundary operator in a chain complex. 

We can express the Laplacian matrix $L_G$, as defined above, as $F F^t$, but we must define $F$ and $F^*$ suitably.

We begin by directing the edges of the graph $G$ arbitrarily.  As above, $v_1, \ldots, v_n$ are the vertices of $G$. Let  $e_1, \ldots, e_m$ denote the edges. Then $F=(f_{i,j})$ is an $n \times m$ matrix with entries in the ring $\Z[i][x_1^{\pm 1}, y_1^{\pm 1}, \ldots, x_g^{\pm 1}, y_g^{\pm 1}]$.  For each $i, j$, the entry $f_{i,j}$ is equal to $\sqrt{w_j \phi(e_j)}$ if $e_j$ has initial vertex $v_i$; 
$f_{i,j} = -\sqrt{w_j \phi(e_j)}$ if $e_j$ terminates at $v_i$; otherwise, $f_{i,j}=0$. Note that since $w_j = 1, -1$, the value of 
$\sqrt{w_j}$ is either 1 or $i$.  

The transpose matrix $F^*$ is defined as usual but we replace all connections $\phi(e_j)$ with their inverses.

\end{remark}

\begin{example} If $G$ is a 1-vertex graph, then its Laplacian polynomial has the form 
\begin{equation}\label{1vertex}  \sum \s_e (2-\phi_e - \phi_e^{-1}), \end{equation}
where the summation is over the edges (loops) $e \in E_G$.  \end{example} 

\begin{prop} Let $e$ be a non-loop edge of G. Then 
\begin{equation}\label{skein}  \De_G =  \De_{G\setminus e} + \s_e \De_{G/e}. \end{equation} \end{prop}

\begin{proof} 

We make use of a result of Forman \cite{Fo93} that expresses $\De_G$ in terms of certain subgraphs of $G$. A \textit{cycle-rooted spanning forest} (CRSF) is a subgraph $F$ containing all the vertices of $G$ and such that each component has a unique cycle. Then

\begin{equation}\label{formaneq} \De_G= \sum_F \prod_{e \in E_F} \s_e \prod_{{\rm cycles\ of\ }F} (2-\phi - \phi^{-1}), \end{equation}
where the summation is over all cycle-rooted spanning forests $F$ of $G$, and $\phi^{\pm 1}$ are the connections of the two orientations of the cycle.

For $e$ a fixed non-loop edge of $G$, we partition the CRSFs of $G$ into those that contain $e$ and those that do not. These are in one-to-one correspondence with the CRSFs of $G/e$ and $G\setminus e$, respectively. In the second case the weight 
$\prod_{e' \in E_F} \s_{e'}$ is unchanged, but in the first case the factor $\s_e$ is lost. \end{proof}

Proposition \ref{skein} enables us to compute $\De_G$ using a skein computation tree. Leaves of the tree are disjoint unions of 1-vertex graphs. The following are easily proved and useful for computation.

\begin{itemize} \item{} If $G$ is the union of two disjoint subgraphs $G_1$ and $G_2$, then  $\De_G = \De_{G_1} \De_{G_2}$. 
\item{} If $G$ has no edges, then $\De_G = 0$.
\item{} If $G$ is a cyclic graph, then $\De_G = \s (2- \phi -\phi^{-1})$, where $\s$ is the product of the signs $\s_e$ of edges of $G$, and $\phi$ is the connection of the cycle (with either orientation).  \end{itemize}

\begin{example} \label{thetaex} Consider the theta-graph $G$ of Figure \ref{theta}, embedded in the torus. (Recall that unlabeled edges have positive sign.) Deleting the middle, vertical edge $e$, produces a 2-vertex cyclic graph $G\setminus e $. Contracting  it results in a 1-vertex graph. By equation (\ref{skein}) we have 
$$\De_G = (2- xy^{-1}-x^{-1}y) + (4-x-x^{-1}-y-y^{-1}) = 6 - x -x^{-1}-y-y^{-1}-xy^{-1}-x^{-1}y.$$
The polynomial can also be computed directly from the Laplacian matrix $L_G$.
\begin{equation}\label{dual} L_G = \begin{pmatrix} 3& -1-x^{-1}-y^{-1} \\ -1 -x -y & 3\\ \end{pmatrix} \end{equation}
\end{example}

\begin{figure}
\begin{center}
\includegraphics[height=1.5 in]{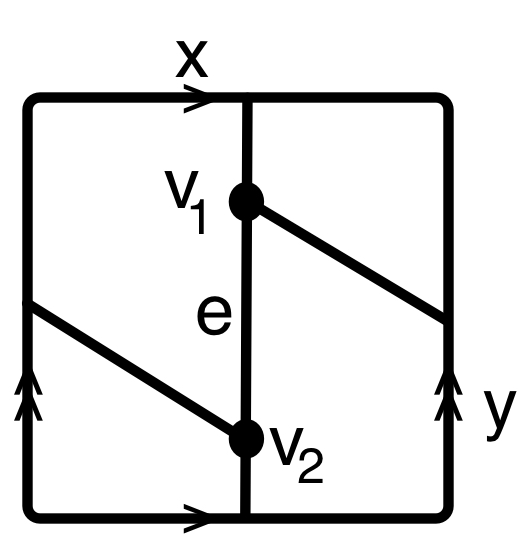}
\caption{Theta-graph $G$ embedded in a torus}
\label{theta}
\end{center}
\end{figure}

The Laplacian matrix $L_G$ determines a homomorphism of the free module $\La^n$ generated by the vertices of $G$. 
The cokernel $\La^n/ \La^n L_G$ is a $\La$-module $\L_G$, the \textit{Laplacian module} of $G$. Its module order   
is the Laplacian polynomial $\De_G$. In this context it is well defined only up to multiplication by units in $\La$. However, Forman's equation (\ref{formaneq}) provides a normal form that is preserved by equation \ref{skein}. (Definition \ref{lapdef} also gives the normalized form since it agrees with Forman's formula.) We will henceforth assume that $\De_G$ is normalized in this way, well defined up to multiplication by $-1$. 

\begin{remark} 
When $G \subset \S^2$, all connections are trivial and the Laplacian module $\L_G$ is a finite abelian group. If furthermore all edge weights of $G$ are $+1$, then $\L_G$ is the direct sum of $\Z$ and the \textit{abelian sandpile group} of $G$. See \cite{CP18} for details. \end{remark}


\section{Applications to links in surfaces}

A \textit {thickened surface} is a product $S \times [0,1]$, where $S$ is a closed, connected oriented surface. For $p \in S$, we think of the point $(p,1)$ as lying above $(p,0)$. 
A \textit{link} $\ell \subset S \times [0,1]$ is a pairwise disjoint collection of finitely many embedded closed curves. It can be described by a \textit{link diagram} $D$: a 4-valent graph $|D|$ embedded in $S$, called a \textit{universe}, with hidden-line effect in a neighborhood of each vertex indicating how one strand of $\ell$ passes over another. 

\begin{figure}
\begin{center}
\includegraphics[height=1.5 in]{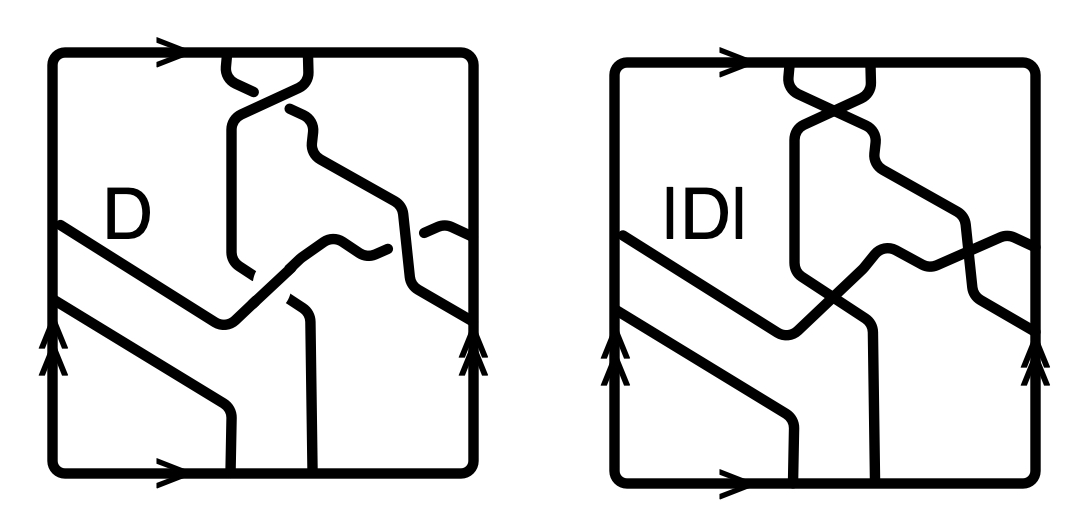}
\caption{Diagram $D$ and universe $|D|$}
\label{diagram}
\end{center}
\end{figure}

\textit{R-equivalence} of link diagrams in $S$ is the equivalence relation generated by Reidemeister moves (See Figure \ref{Rmoves}.) It is well known that Reidemeister's original proof in \cite{Re48} for planar links shows generally that two links in $S \times [0,1]$ are isotopic if and only if their diagrams are R-equivalent. 

A \textit{region} of a link diagram is a connected component of $S \setminus |D|$. A link diagram is \text{cellular} if all of its regions are contractible. We can turn any diagram into a cellular diagram with Reidemeister moves of type II.
Less obvious is the following.

\begin{prop} \label{cellular} If two cellular link diagrams $D, D' \subset S$ are equivalent by Reidemeister moves then they are related by a sequence of cellular diagrams such that each diagram is obtained from the previous one by a single Reidemeister move. \end{prop} 

\begin{proof} Assume that $D$ and $D'$ are cellular link diagrams that are equivalent by a sequence of Reidemeister moves. 
The only ``bad move" that can destroy the contractibility of a region is the forward direction of a type II move in Figure \ref{Rmoves}.
A bad move opens a channel from a region to itself. 
Before any bad move is performed, we add a thin finger as in Figure \ref{2move}, bridging the channel and thereby protecting the region's contractibility. The tip of the finger is free but it should follow the diagram as it is deformed, staying above the arc below its tip. We perform the sequence Reidemeister moves that brings $D$ to $D'$. The fingers will be stretched, twisted, and might even cross over each other but no arc should pass through a finger's interior.  Finally we retract the fingers by type II moves. Since the $D'$ is cellular, non-contractible regions will not appear during this final phase, a claim that is easy to see by imagining the retractions in reverse.

\begin{figure}
\begin{center}
\includegraphics[height=2.5 in]{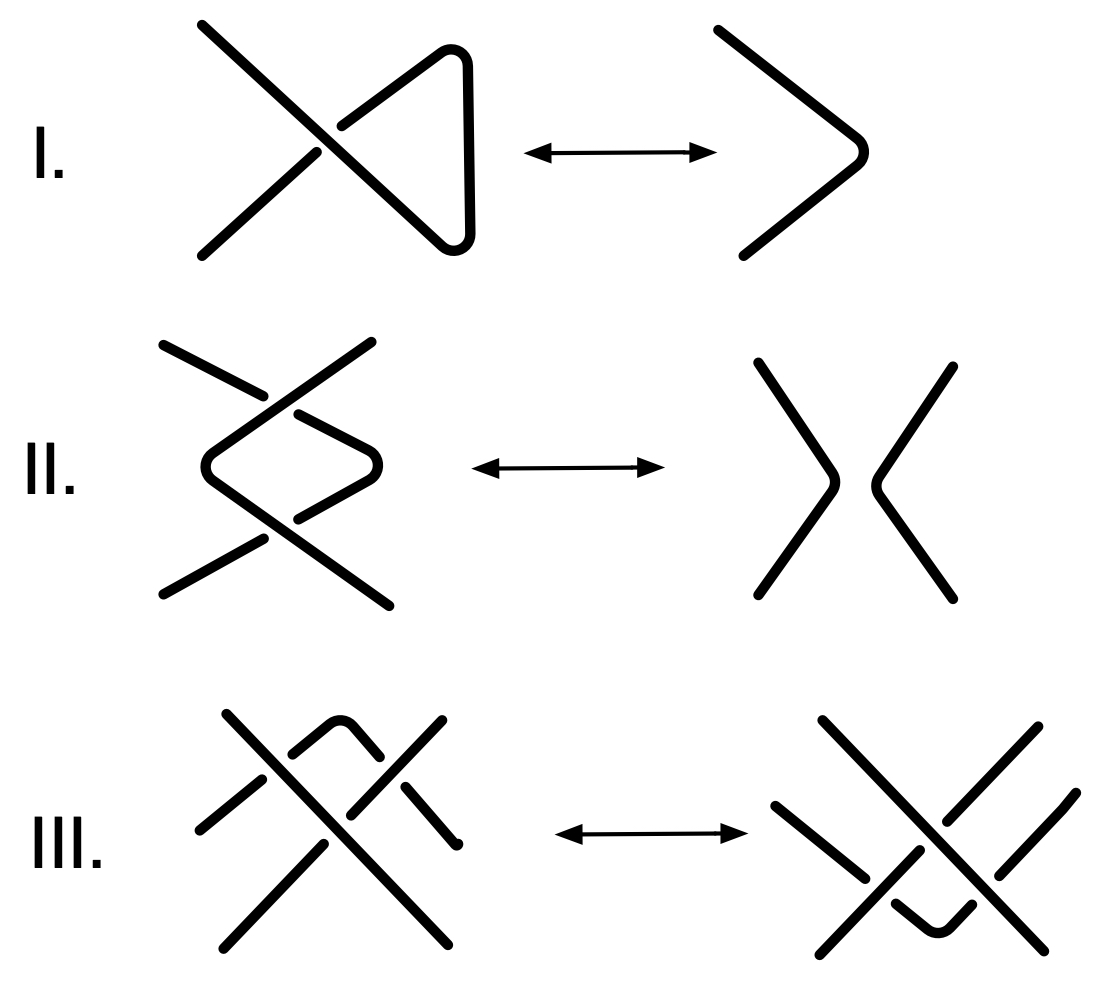}
\caption{Reidmeister moves}
\label{Rmoves}
\end{center}
\end{figure}

\begin{figure}
\begin{center}
\includegraphics[height=1in]{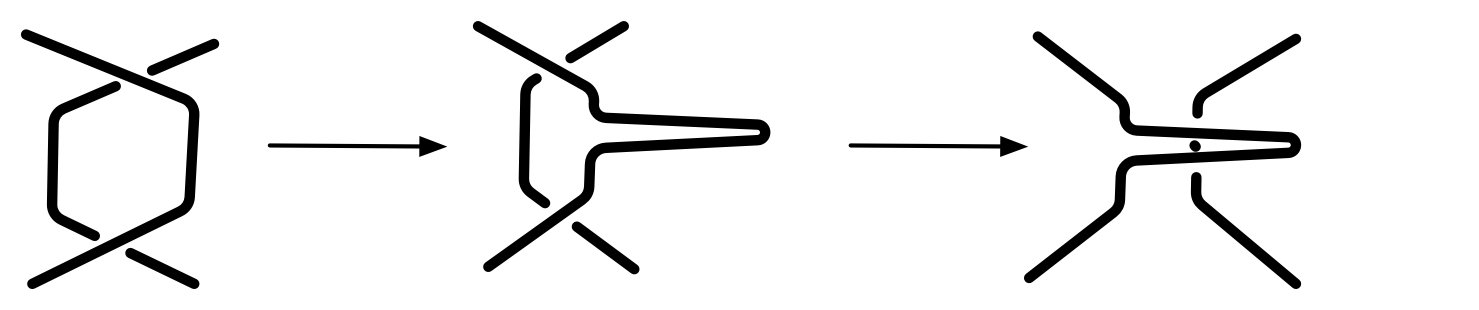}
\caption{Adding a finger before second Reidemeister move. }
\label{2move}
\end{center}
\end{figure}
\end{proof}

A link diagram $D$ is \textit{checkerboard colorable} if it is possible to shade certain regions of $D$ so that whenever two regions share a boundary, exactly one of them is shaded. A diagram with such a shading is \textit{checkerboard colored}. (We will denote it also by $D$ in order to avoid excessive notation.) Any checkerboard colorable diagram admits exactly two such shadings. The proof of the following is easy and left to the reader.

\begin{lemma}\label{CBC} A link diagram $D \subset S$ is checkerboard colorable if and only if the inclusion map $|D|\hookrightarrow S$ induces a trivial homomorphism of homology groups with $\Z/2$ coefficients. \end{lemma} 

\begin{remark} 1. Links that satisfy the condition of Lemma \ref{CBC} were called \textit{mod-2 almost classical} in \cite{BGHW17}. The condition is equivalent to the link having a diagram admitting a mod-2 Alexander numbering \cite{CESW09}. 

2. Any Reidemeister move takes one checkerboard colored diagram to another (while changing the diagrams only locally). 
 If both diagrams are cellular, then their associated graphs are related by the \textit{Reidemeister graph moves} 
 in Figure \ref{RGmoves}. (See \cite{YK57} or \cite{Ka06}.)

3. Given a checkerboard colorable diagram for a link $\ell \subset S \times [0,1]$, the two shaded diagrams $D, D^*$ have a homological interpretation: Let $N$ be a tubular neighborhood of $\ell$, and consider the section of the long exact sequence of homology groups with $\Z/2$ coefficients that corresponding to the pair $N \subset S \times [0,1]$: 
\begin{equation}\label{exactseq} 0\to \Z/2 \to H_2(S \times [0,1], N) \to H_1(N) \end{equation}
The shaded diagrams $D, D^*$ can be seen as  spanning surfaces of $\ell$; they represent elements $[D], [D*] \in H_2(S \times [0,1], N)$.  The classes of meridians $m_1, \ldots, m_d$ of the components of $\ell$ form a basis for $H_(N) \cong (\Z/2)^d$. By exactness of the sequence (\ref{exactseq}), the \textit{total class} $m= [m_1] + \ldots +[m_d]$ has precisely two preimages in $H_2(S \times [0,1], N)$. Clearly $[D]$ and $[D^*]$ are preimages. Moreover, they are distinct since $D \cup D^* = S$ and hence $[D]+[D^*]$ generates $H_2(S \times [0,1]) \cong H_2(S) \cong \Z/2$. 

Reidemeister moves can interchange the two homology classes. In fact, if $S = \S^2$, then this is always possible \cite{YK57}. However, for surfaces of higher genus this is not generally the case (see Example \ref{thetaex} below.) 

\end{remark}

A checkerboard colored cellular link diagram $D \subset S$ determines a signed embedded graph $G \subset S$, unique up to isotopy, via the ``medial graph" construction (see Figure \ref{medial}). Vertices of $G$ correspond to shaded regions of $D$, while each pair of vertices corresponding to shaded regions meeting at a crossing of $D$ are joined by an edge $e$ with weight $w_e$ determined as in Figure \ref{medial}. From the graph we can reconstruct $D$. The other checkerboard coloring of $D$ determines a dual signed graph $G^*$. If $e^*$ is an edge of $G^*$ dual to $e$, then $w_{e^*} = -w_e$. 

\begin{figure}
\begin{center}
\includegraphics[height=1.3 in]{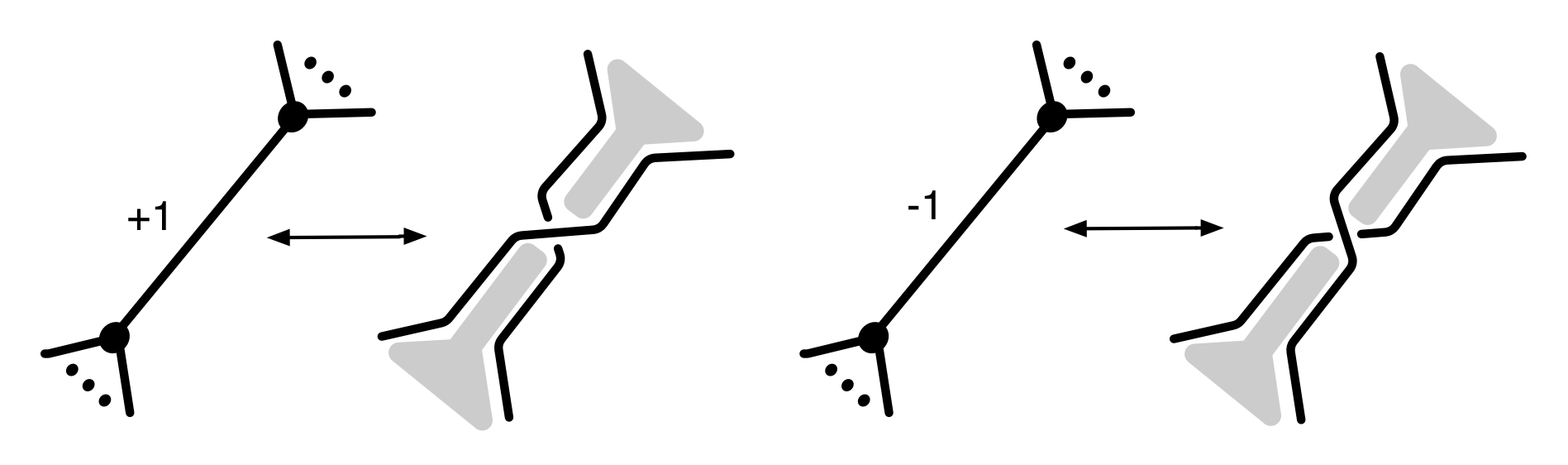}
\caption{Constructing a checkerboard colored link diagram from a medial graph}
\label{medial}
\end{center}
\end{figure}

\begin{figure}
\begin{center}
\includegraphics[height=4 in]{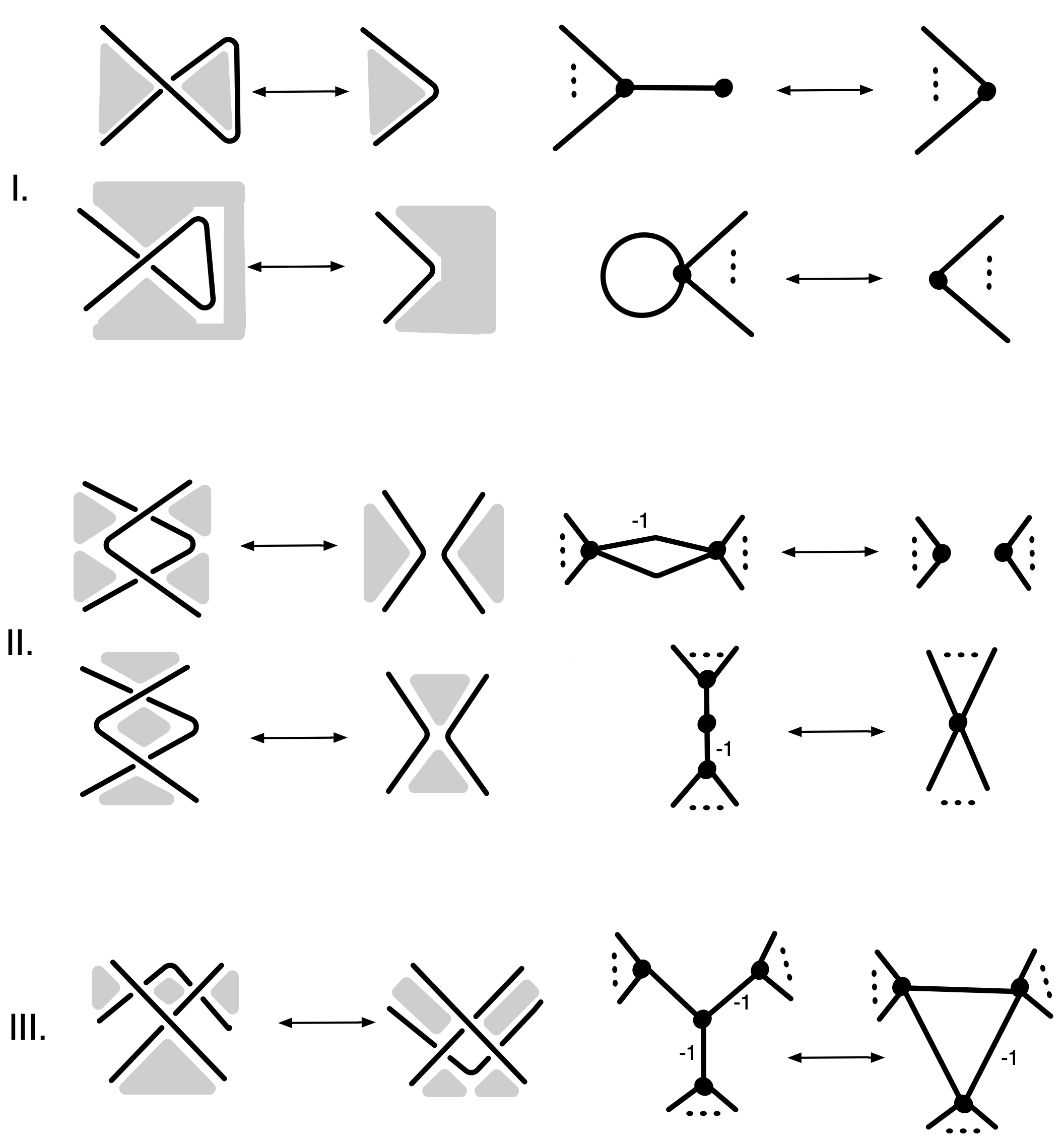}
\caption{Reidmeister moves and Reidemeister graph moves}
\label{RGmoves}
\end{center}
\end{figure}

\begin{prop} \label{module} Assume that $G, G'$ are signed graphs (not necessarily embedded) in a closed, connected orientable surface $S$.  If $G'$ is the result of applying a Reidemeister graph move to $G$, then the $\La$-modules $\L_G, \L_{G'}$ are isomorphic.\end{prop}

\begin{proof} There are three moves to check, with multiple cases depending on shading and crossing signs. We check two representative cases, and leave the rest to the reader. 

We examine the first Reidemeister graph move. Vertices $v, w$ in Figure \ref{reidemeister1}(a) contribute the relations 
\begin{equation} (E_v+1)v = {\Sigma}_ v + \phi_e w, \quad w = \phi^{-1}_e v, \end{equation}
where $E_v$ is the sum of signs of edges incident to $v$ but not $w$ (with loops counted twice), ${\Sigma}_v$ is the $\La$-linear combination of vertices other than $w$ that are joined by edges to $v$, and $e$ is the directed edge from $v$ to $w$. We use the second relation to eliminate $w$. The first relation becomes $(E_v)v = {\Sigma}_v$, which is the relation corresponding to the vertex $v$ in Figure \ref{reidemeister1}(b). Hence $\La_G$ is unchanged by the move.

Next we consider the third Reidemeister graph move.  The vertices $v, w_1, w_2, w_3$ in Figure \ref{move2}(a) contribute the  relations (where symbols have meaning similar to those in the previous case):
\begin{enumerate} 
\item[] $-v = \phi_{v, w_1} w_1 - \phi_{v, w_2} w_2 - \phi_{v, w_3}w_3$ 
\item[] $(E_{w_1}+1)w_1 = {\Sigma}_{w_1} + \phi^{-1}_{v, w_1}v$
\item[] $(E_{w_2}-1)w_2 = {\Sigma}_{w_2} - \phi^{-1}_{v, w_2}v$
\item[] $(E_{w_3}-1)w_3 = {\Sigma}_{w_3} - \phi^{-1}_{v, w_3}v$
\end{enumerate}
We use the first relation to eliminate $v$. The three remaining relations  become:
\begin{enumerate}
\item[] $(E_{w_1}+2)w_1 = {\Sigma}_{w_1} + \phi_{w_1, w_2} w_2 + \phi_{w_1, w_3} w_3$
\item[] $E_{w_2} w_2 = {\Sigma}_{w_2} + \phi_{w_2, w_1} w_1 - \phi_{w_2, w_3} w_3$
\item[] $E_{w_3} w_3 = {\Sigma}_{w_3} + \phi_{w_3, w_1} w_1 - \phi_{w_3, w_2} w_2,$
\end{enumerate}
which are the relations corresponding to $w_1, w_2, w_3$ in Figure \ref{move2}(b).
\end{proof} 

\begin{figure}
\begin{center}
\includegraphics[height=1.3 in]{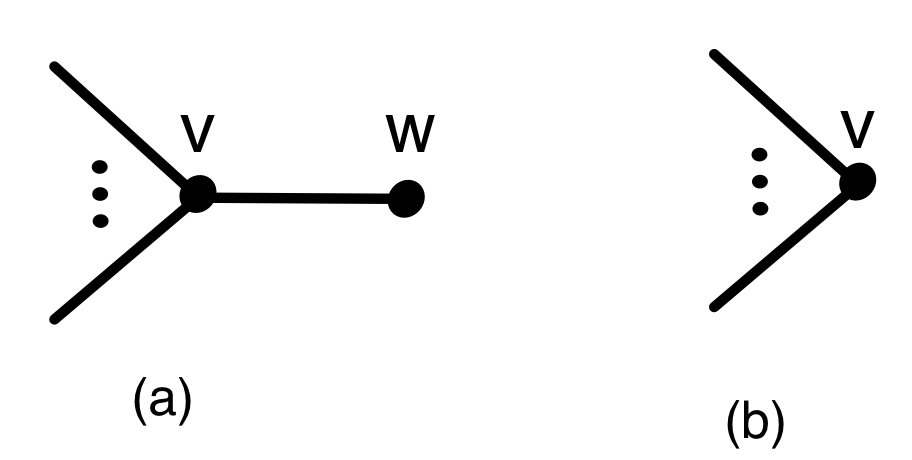}
\caption{Graph transformation corresponding to first Reidemeister move}
\label{reidemeister1}
\end{center}
\end{figure}

\begin{figure}
\begin{center}
\includegraphics[height=2 in]{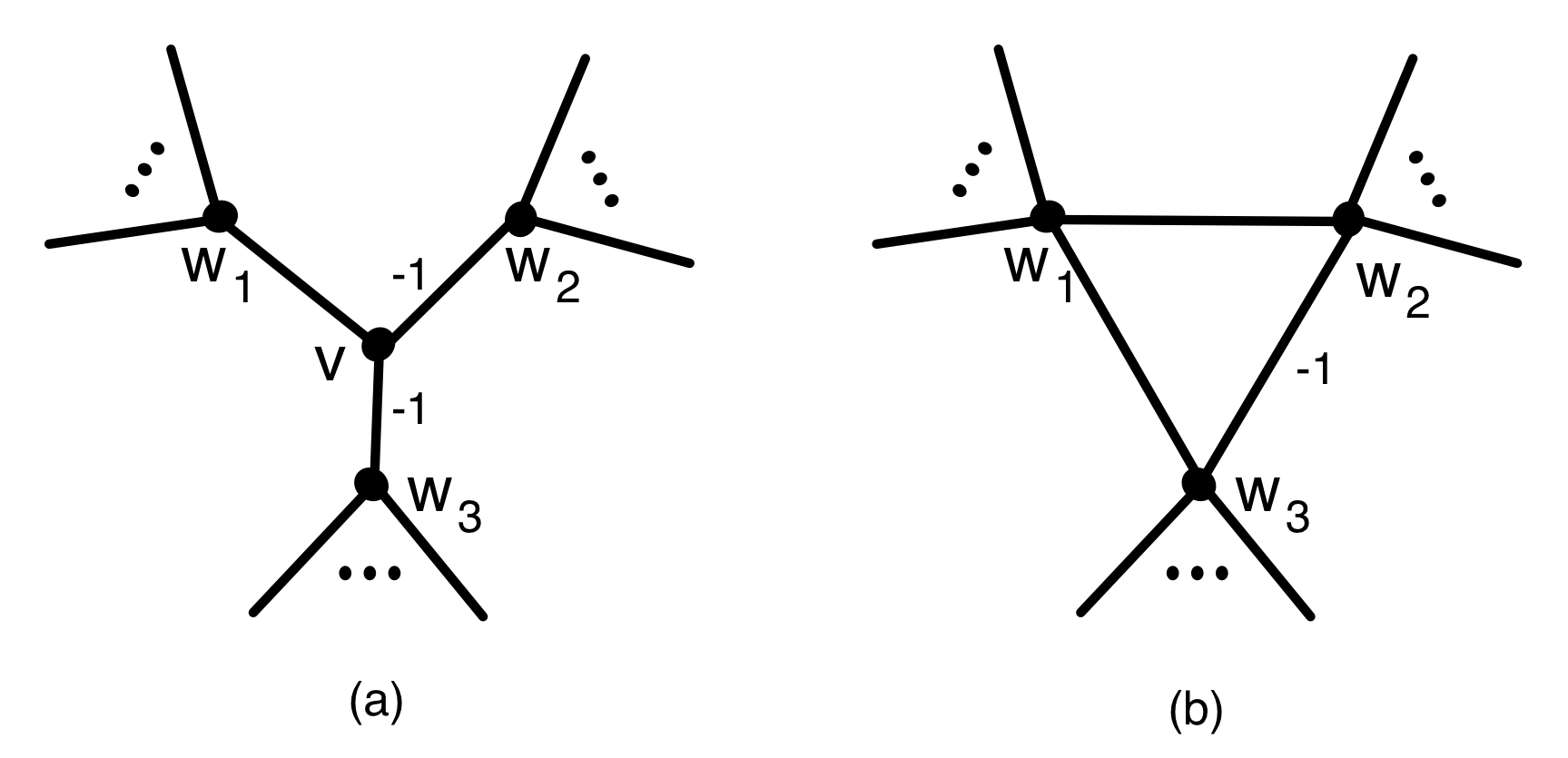}
\caption{Graph transformation corresponding to third Reidemeister move}
\label{move2}
\end{center}
\end{figure}


\begin{cor} \label{pairmodule} Assume that $G$ is a signed embedded graph associated to a cellular checkerboard colored diagram of a link $ \ell \subset S \times [0,1]$. The pair $\{\L_G, \L_{G^*}\}$ is an invariant of $\ell$.  \end{cor} 

\begin{proof} Assume $G$ is associated to $D$. If $D'$ is another cellular checkerboard colored diagram, then by Proposition \ref{cellular} there is a sequence of Reidemeister moves taking $D'$ to $D$ and such that each intermediate diagram is cellular. The moves transform $G'$ to either $G$ or $G^*$ (see Remark \cite{CBC}). Proposition \ref{module} implies that $\L_{G'}$ is isomorphic to $\L_{G}$ or $\L_{G^*}$.
\end{proof}

Recall that we take the normalized Laplacian polynomial $\De_G$ is well defined up to multiplication by $-1$. 

\begin{cor} \label{pairpoly} Assume that $G$ is a signed embedded graph associated to a cellular checkerboard colored diagram of a link $ \ell \subset S \times [0,1]$. The pair $\{\De_G,   \De_{G^*}\}$ is an invariant of $\ell$.  \end{cor}

\begin{remark} \label{dependence} 1. Like $\L_G$, the polynomial $\De_G$  depends on the symplectic basis we chose for $H_1(S; \Z)$. If we regard $\De_G$ up to possible change of symplectic basis, then $\{\De_G, \De_{G^*}\}$ is an invariant of $\ell$ up to automorphisms of $S$. 

2. When $S = \S^2$, the matrix $L_G$ is a Goeritz matrix of $\ell$ (see \cite{STW19}). If we regard $\ell$ as a link in the 3-sphere, then $\L_G$ is isomorphic to $\Z \oplus H_1(M_2;\Z)$, where $M_2$ is the 2-fold cover of $\S^3$ branched over $\ell$. 
\end{remark}

\begin{figure}
\begin{center}
\includegraphics[height=2 in]{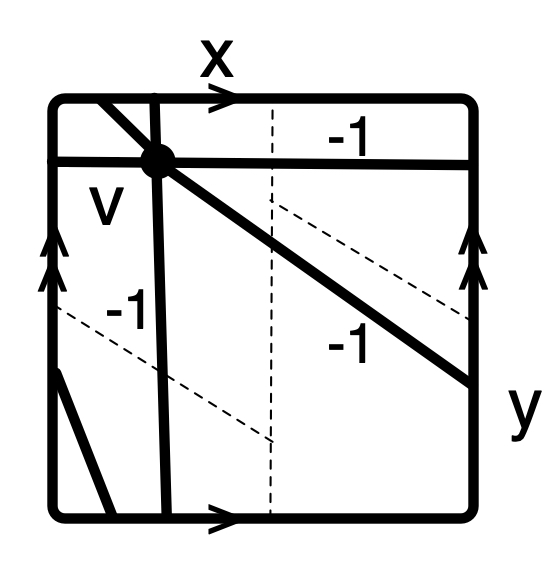}
\caption{The dual $G^*$ of a theta graph}
\label{thetadual}
\end{center}
\end{figure}

Let $D \subset \S^2$ be a checkerboard colored link diagram, and let $D^*$ denote the dual checkerboard colored diagram.
By stretching an outer arc of $D$ and pulling over the rest of the diagram and then around the sphere to its original position, we see that $D$ and $D^*$ are equivalent under Reidemeister moves. (This fact was first observed in  \cite{KY57}, where a more formal argument appears.) Hence  $\L_G$ and $\L_{G^*}$ are isomorphic.

The following example shows that for diagrams in surfaces of higher genus, $\L_G$ and $\L_{G^*}$ need not be isomorphic. 

\begin{example} \label{thetaex} The Laplacian module of the graph $G$ in Figure \ref{theta} has a presentation
$$\L_G \cong \langle v_1, v_2 \mid 3 v_1 = (1 + x+ y) v_2, \quad 3 v_2 = (1+ x^{-1}+ y^{-1})v_1 \rangle.$$
The dual graph $G^*$, which appears in Figure \ref{thetadual}, has a single vertex. Its Laplacian module is 
$$\L_{G^*} \cong \< v \mid 6 v = (x+x^{-1} + y + y^{-1} + x y^{-1} + x^{-1} y)v \rangle.$$

To see that the modules are not isomorphic, we reduce the ring $\La$ to $\Z$, setting $x, y$ equal to $1$.
(More rigorously, we treat $\Z$ as a trivial right $\La$-module and pass to the tensor product modules $\Z \otimes_\La \L_G$ and $\Z\otimes_\La  \L_{G^*}$.) The resulting abelian groups are $\Z \oplus \Z/3$ and $\Z$, respectively. Since they are not isomorphic, neither are $\L_G$ and $\L_{G^*}$. 
\end{example}

The Laplacian polynomials $\De_G$ and $\De_{G^*}$ in Example \ref{thetaex} are the same. This holds generally for signed graphs   in the torus \cite{SW19}, a consequence of the fact that null-homologous curves in the torus are contractible. The following example shows that equality does not hold for surfaces of higher genus.

\begin{example} Consider the signed graph $G$ embedded in the surface $S$ of genus 2 in Figure \ref{genus2}. With respect to the indicated symplectic basis for $H_1(S; \Z)$ the Laplacian module $\L_G$ has presentation matrix

\begin{equation} \begin{pmatrix} 5-x -x^{-1} -y -y^{-1} & -1\\ -1 & 5 -u -u^{-1} -v -v^{-1}\\  \end{pmatrix} \end{equation}

and $$\De_G = 24 - 5(x+x^{-1}+y+y^{-1} +u + u^{-1} + v + v^{-1}) + $$ $$ux + u^{-1}x^{-1} + ux^{-1} + u^{-1}x + v x + v^{-1}x^{-1}+ v x^{-1}+v^{-1}x+ $$ $$ uy + u^{-1} y^{-1} + u y^{-1} + u^{-1}y + v y + v^{-1} y^{-1} + v y^{-1}+v^{-1} y.$$

The dual graph $G^*$ has a single vertex, loops (suitably directed) with connections $x, y, u, v$ , and a fifth loop (transverse to $e$)  that is null-homologous and so has trivial connection. Hence
$$\De_{G^*} = 8 -x - x^{-1} - y - y^{-1} -u -u^{-1} - v - v^{-1}.$$

\begin{figure}
\begin{center}
\includegraphics[height=2.5 in]{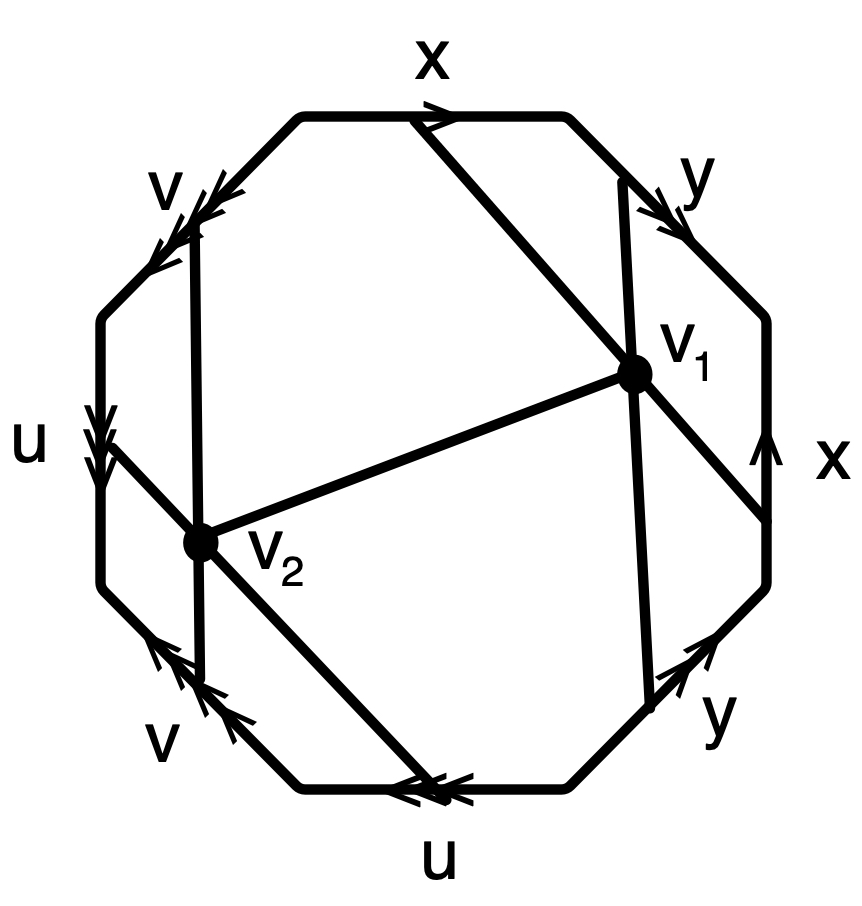}
\caption{Graph $G$ in a surface of genus 2 such that $\De_G$ and $\De_{G^*}$ are different}
\label{genus2}
\end{center}
\end{figure}

\end{example} 

\begin{example} The theta graph in Example \ref{thetaex} corresponds to a 3-component link $\ell_1$ in the torus. Placing the weight $-1$ on one of the three edges results in three  links $\ell_1, \ell_2, \ell_3$. The  links and their graphs  $G_1, G_2, G_3$ are shown in Figure \ref{three}. The corresponding Laplacian polynomials with respect to the indicated symplectic basis are: 
\begin{itemize} 
\item{} $\Delta_{G_1}= -2 + (x + x^{-1}) + (y+y^{-1}) -(x^{-1}y + xy^{-1})$
\item{} $\Delta_{G_2}= -2 + (x + x^{-1}) - (y+y^{-1}) + (x^{-1}y + xy^{-1})$
\item{} $\Delta_{G_3}= -2  -(x + x^{-1}) + (y+y^{-1}) +(x^{-1}y + xy^{-1})$
\end{itemize}

\noindent Consequently, the links are pairwise non-isotopic. 

By considering fundamental domains in the universal cover of $S$, one can readily see that there is an order-3 automorphism $f$ of the $S$ sending $x$ to $y^{-1}$ and $y$ to $x y^{-1}$, where $x$ and $y$ are generators of $\pi_1(S; \Z)$, and mapping the diagram of $\ell_1$ (resp. $\ell_2$) to that of $\ell_2$ (resp. $\ell_3$). 
Moreover, $f$ induces a change of symplectic basis of $H_1(S; \Z)$ mapping $x \mapsto -y, y \mapsto x-y$ that transforms $\De_{G_1}$ (resp. $\De_{G_2}$) to $\De_{G_2}$ (resp. $\De_{G_3}$). 

\end{example}

\begin{figure}
\begin{center}
\includegraphics[height=2.8 in]{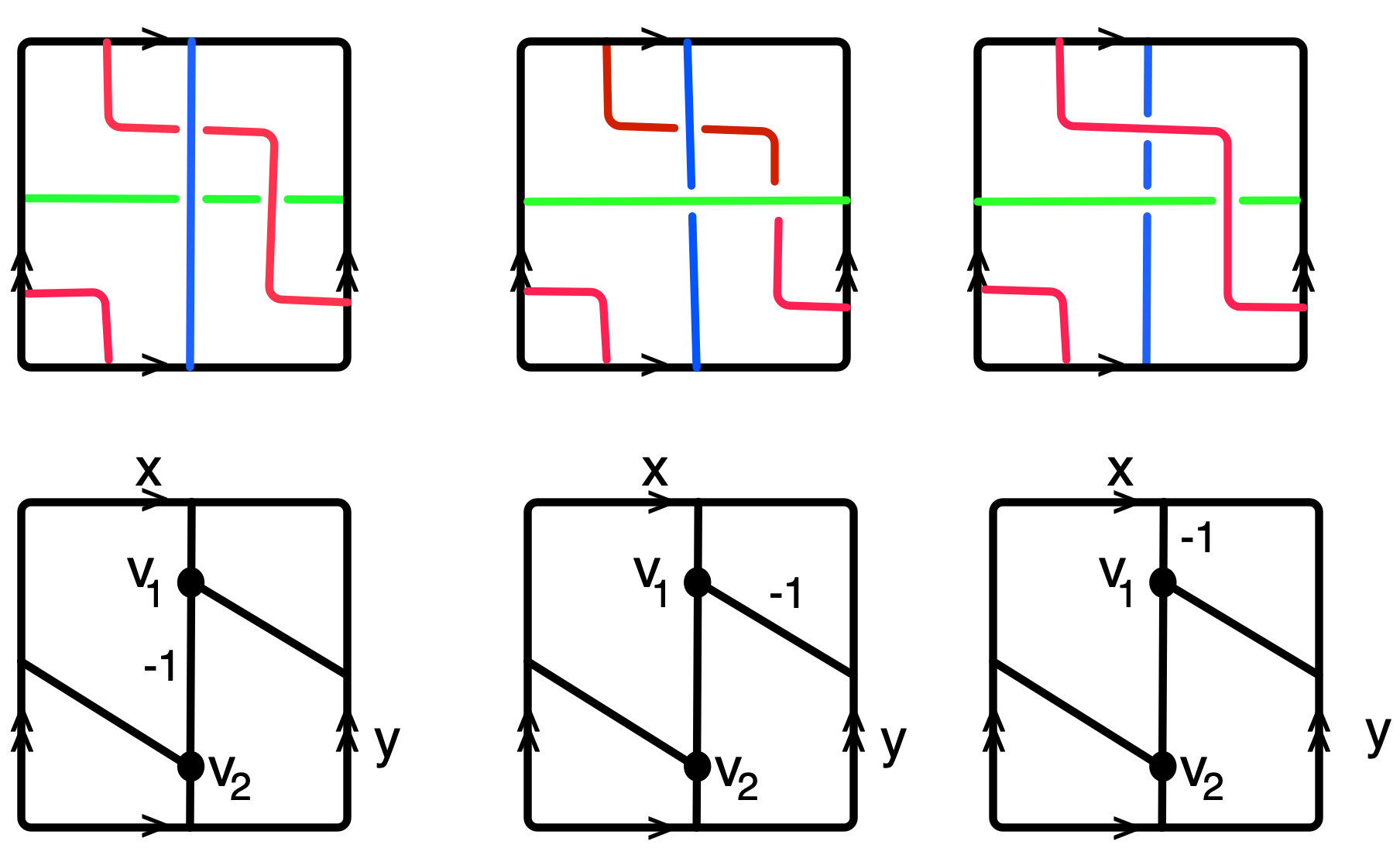}
\caption{3-component links $\ell_1, \ell_2, \ell_3$ and their graphs $G_1, G_2, G_3$ (left to right)}
\label{three}
\end{center}
\end{figure}

\section{Virtual links and genus} \label{genus}

A \textit{virtual link} $\ell$ can be regarded as a link diagram $D$ defined up to Reidemeister moves, surface automorphisms and addition or deletion of hollow handles. Handle addition is 0-surgery, replacing $\S^0 \times \D^2 \subset S \setminus D$ with $\D^1 \times \S^1$, thereby increasing the genus of the surface. The handle deletion is 1-surgery along a simple closed curve $C \subset S \setminus D$, replacing $C \times \D^1$ with $\D^2 \times \S^0$; when $C$ is essential and nonseparating, the surgery decreases the genus of the surface. The \textit{virtual genus} $vg\ (\ell)$ is the minimum possible genus $g\ (S)$ a surface $S$  that can be obtained in this way. 

The information in the polynomial $\De_G$ can often establish that $g(S) = vg(\ell)$. Our approach is similar to that of \cite{CSW13}. 
Assume that $g(S) > vg(\ell)$. Then a theorem of G. Kuperberg \cite{Ku03} implies that after Reidemeister moves $S \setminus |D|$ contains an essential nonseparating simple closed curve $C$ upon which we can perform surgey to reduce the genus of $S$. If we orient $C$ and extend its homology class to a symplectic basis for $H(S; \R)$, then this basis element will not appear among the coefficients of $\De_G$. Contrapositively, if the coefficients of $\De_G$ span $H_1(S; \R)$, then $g(S) = vg(\ell)$. We make this more precise in the following. The same conclusion holds for the dual graph $G^*$.

\begin{definition} The \textit{symplectic rank} $rk_s(\De_G)$ is the rank of the submodule of $H_1(S; \R)$ generated by the summands of $\De_G$. \end{definition}

\begin{theorem} \label{srank} Assume that $G \subset S$ is a signed embedded graph associated to a checkerboard colorable link $\ell \subset S \times [0,1]$, and $G^*$ is the dual graph.  If either $rk_s(\De_G)$ or $rk_s(\De_{G^*})$ is equal to twice the genus of $S$, then $g (S) = vg (\ell)$. \end{theorem}

\begin{example} Consider the virtual link $\ell_{k,l,m}$ in Figure \ref{three}, where $k,l,m$ denote the number of positive or negative twists at the indicated site. The link $\ell_{1,1,1}$ corresponds to the graph in Figure \ref{theta}. The graph of the $\ell_{k,l,m}$ is obtained from $G$ by subdividing edges and appending signs $-1$ if the associated integer is negative. The dual graph $\G^*$ can be obtained from the graph of Figure \ref{thetadual} by replacing edges with multiple edges and appropriate signs. Hence
$$\De_{G^*} = 2(k+l+m) - k(x+x^{-1}) - l(y+y^{-1}) - m(x^{-1}y+xy^{-1}).$$
By Theorem \ref{srank}, the virtual genus of $\ell_{k,l,m}$ is equal to 1.  \end{example}

\begin{figure}
\begin{center}
\includegraphics[height=2 in]{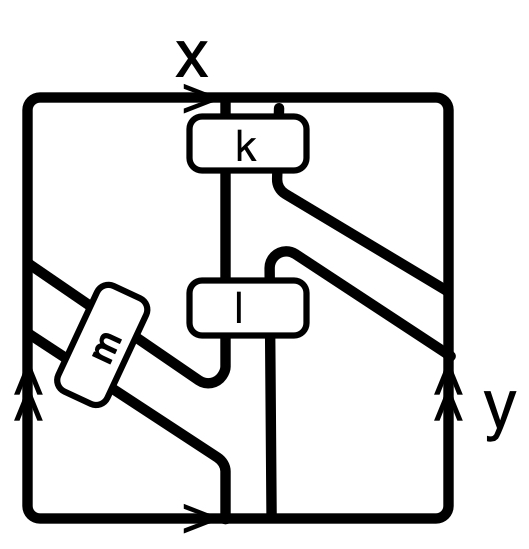}
\caption{Virtual links $\ell_{k,l,m}$}
\label{family}
\end{center}
\end{figure}

\begin{example} If a  diagram $\hat D \subset S$ of a virtual link $\hat \ell$ is not checkerboard colorable, then information about virtual genus of $\hat \ell$ might be obtained by considering a ``doubled diagram" $D$ that describes a satellite link $\ell$ with companion $\hat \ell$.   By \cite{SW13} the virtual genus of $\ell$ is the same as that of $\hat \ell$. 

Figure \ref{satellite} (a) is a diagram of a ``virtual trefoil." Figure \ref{satellite} (b) is a diagram $D$ of a satellite knot $\ell$. The reader can check that the Laplacian polynomial $\De_G$ of the signed graph $G$ associated to $D$ is $4-(x+x^{-1})-3(y+y^{-1})+2(x y^{-1}+x^{-1}y)$. Since the symplectic rank $rk_s(\De_G)$ is 2, Theorem \ref{srank} implies the known result that the virtual trefoil has genus 1. \end{example}

\begin{figure}
\begin{center}
\includegraphics[height=2 in]{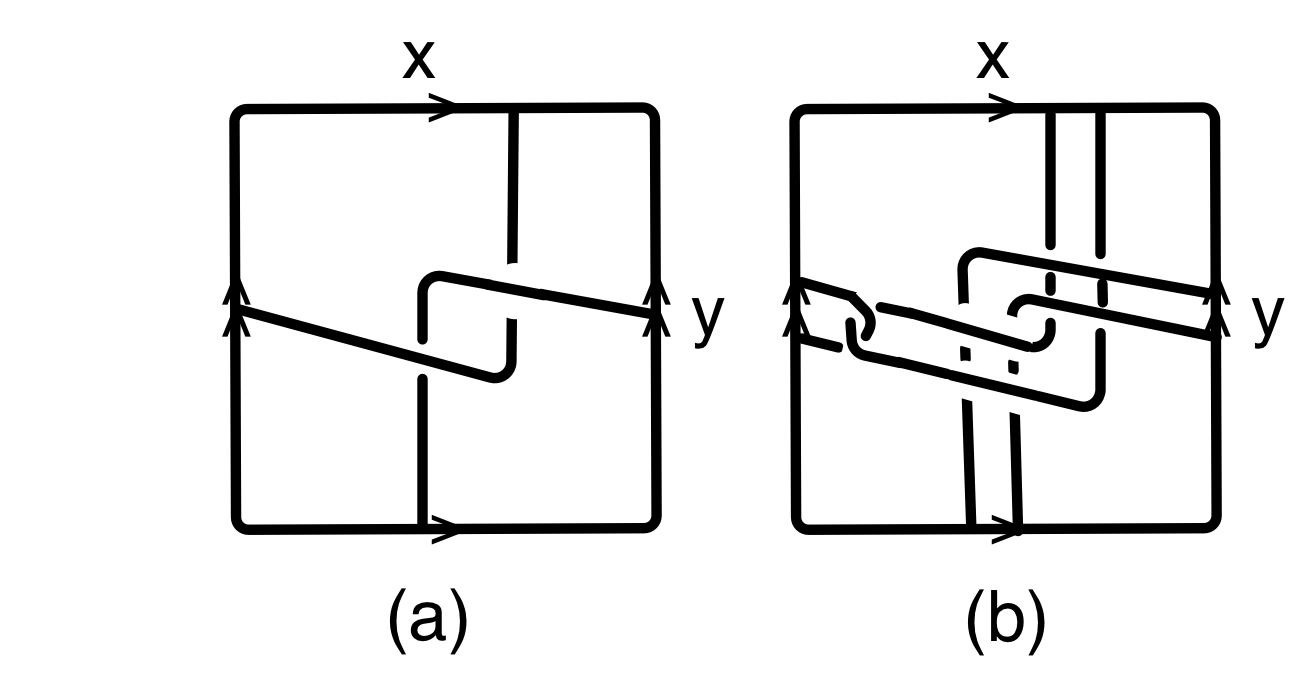}
\caption{Virtual trefoil and satellite}
\label{satellite}
\end{center}
\end{figure}

\begin{example} Consider the 3-component virtual link $\ell$ described by the graph $G$ in Figure \ref{challenge}(a). It is easy to see that $\De_G = 0$. Hence Theorem \ref{srank} gives no information about $vg(\ell)$ in this case. The link can be drawn as in Figure \ref{challenge} (b). 

We give  a direct proof that $vg(\ell)=2$. 
Lift the link diagram to the universal (abelian) cover $\tilde S$. There we see countably many unknots in the plane. If $\ell$ has genus 0, then, after Reidemeister moves, we can find an essential simple closed curve in the torus $S$ that misses the diagram. Perform the Reidemeister moves equivariantly in $\tilde S$ and consider a single component of the lifted curve. The component separates the plane. However, any two circles of the lifted diagram $\tilde D$ are part of a chain of consecutively-linked circles. Hence $\tilde D$ lies on only one side of the separating curve, which is not possible. Hence $vg(\ell)=1$. 

\begin{figure}
\begin{center}
\includegraphics[height=2 in]{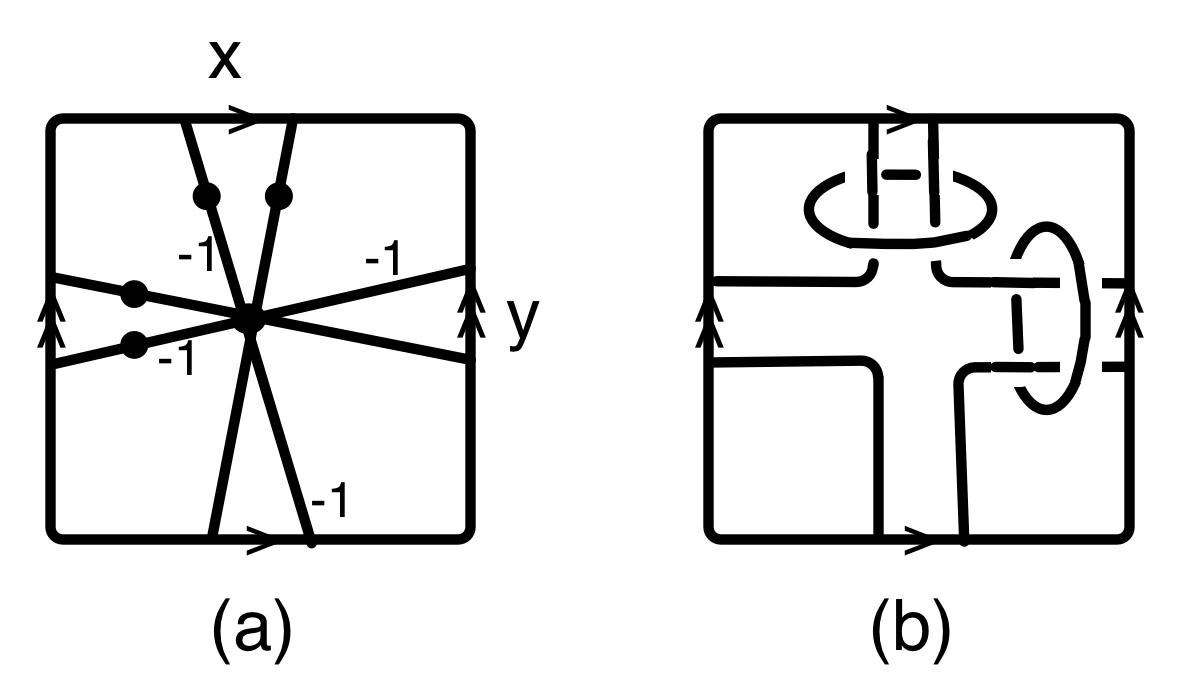}
\caption{Graph $G$ and associated virtual link $\ell$}
\label{challenge}
\end{center}
\end{figure}

\end{example}

\bigskip

\ni Department of Mathematics and Statistics,\\
\ni University of South Alabama\\ Mobile, AL 36688 USA\\
\ni Email: \\
\ni  silver@southalabama.edu\\
\ni swilliam@southalabama.edu

\end{document}